\documentclass[a4paper,english,fontsize=10pt,parskip=half,abstracton]{scrartcl}
\usepackage{babel}
\usepackage[utf8]{inputenc}
\usepackage[T1]{fontenc}
\usepackage[a4paper,left=20mm,right=20mm,top=30mm,bottom=30mm]{geometry}
\usepackage{amsmath}
\usepackage{amsthm}
\usepackage{amssymb}
\usepackage{enumerate}
\usepackage[bookmarks=false,
            pdftitle={On 2-blocks with minimal nonabelian defect groups II}]{hyperref}
\usepackage[arrow, matrix]{xy}
\newtheorem{Theorem}{Theorem}
\newtheorem{Lemma}[Theorem]{Lemma}
\newtheorem{Proposition}[Theorem]{Proposition}
\newtheorem{Corollary}[Theorem]{Corollary}

\renewcommand{\phi}{\varphi}
\newcommand{\C}{\operatorname{C}}
\newcommand{\N}{\operatorname{N}}
\newcommand{\Z}{\operatorname{Z}}
\newcommand{\Aut}{\operatorname{Aut}}
\newcommand{\pcore}{\operatorname{O}}
\newcommand{\GL}{\operatorname{GL}}
\newcommand{\PSL}{\operatorname{PSL}}
\newcommand{\PSU}{\operatorname{PSU}}
\newcommand{\SL}{\operatorname{SL}}
\newcommand{\SU}{\operatorname{SU}}
\newcommand{\GU}{\operatorname{GU}}
\newcommand{\E}{\operatorname{E}}
\newcommand{\F}{\operatorname{F}}
\newcommand{\Syl}{\operatorname{Syl}}

\def\bG{{\bf G}}
\def\bH{{\bf H}}
\def\bL{{\bf L}}
\def\bM{{\bf M}}
\def\bT{{\bf T}}

\mathchardef\ordinarycolon\mathcode`\:  %defines a nice ":="
 \mathcode`\:=\string"8000
 \begingroup \catcode`\:=\active
   \gdef:{\mathrel{\mathop\ordinarycolon}}
 \endgroup

\title{On $2$-blocks with minimal nonabelian defect groups II}
\author{C.\,W.\,Eaton,\\ School of Mathematics,\\University of Manchester,\\Manchester, UK,\\\texttt{charles.eaton@manchester.ac.uk} \and
B.\,Külshammer,\\Mathematical Institute,\\Friedrich Schiller University,\\Jena, Germany,\\\texttt{kuelshammer@uni-jena.de} \and
B.\ Sambale,\\Mathematical Institute,\\Friedrich Schiller University,\\Jena, Germany,\\\texttt{benjamin.sambale@uni-jena.de}
}
\date{\today}

\begin{document}
\frenchspacing
\maketitle
\begin{abstract}\noindent
We determine the structure of $2$-blocks with minimal nonabelian defect groups, by making use of the classification of finite simple groups.
\end{abstract}

In \cite{Sambalemna}, the third author of this paper investigated $2$-blocks $B$ of finite groups whose defect groups $D$ are minimal nonabelian; this means that $D$ is nonabelian but all proper subgroups of $D$ are abelian. In most cases, it was possible to determine the numerical invariants $k(B)$, $l(B)$ and $k_i(B)$, for $i\ge 0$. Here, as usual, $k(B)$ denotes the number of irreducible ordinary characters in $B$, $l(B)$ denotes the number of irreducible Brauer characters in $B$, and $k_i(B)$ denotes the number of irreducible ordinary characters of height $i$ in $B$, for $i\ge 0$.

However, for one family of $2$-blocks only partial results were obtained in \cite{Sambalemna}. Here we deal with this remaining family of $2$-blocks, by making use of the classification of the finite simple groups. Our main result is as follows:

\begin{Theorem}\label{main}
Let $B$ be a nonnilpotent $2$-block of a finite group $G$ with defect group
\begin{equation}\label{defD}
D=\langle x,y: x^{2^r}=y^{2^r}=[x,y]^2=[x,x,y]=[y,x,y]=1\rangle
\end{equation}
of order $2^{2r+1}\ge 32$. Then $B$ is Morita equivalent to $\mathcal{O}[D\rtimes E]$ where $E$ is a subgroup of $\Aut(D)$ of order $3$. In particular, we have
\begin{align*}
l(B)&=3,&k(B)&=\frac{5\cdot 2^{2r-2}+16}{3},&k_0(B)&=\frac{2^{2r}+8}{3},&k_1(B)&=\frac{2^{2r-2}+8}{3}.
\end{align*}
\end{Theorem}

Here $(\mathbb{K},\mathcal{O},\mathbb{F})$ denotes a splitting $2$-modular system for $G$. Let again $D$ be a $2$-group as in \eqref{defD}. If $B$ is a nilpotent $2$-block of a finite group $G$ with defect group $D$, then, by the main result of \cite{Puig}, $B$ is Morita equivalent to $\mathcal{O}D$. So we have the following consequence of Theorem~\ref{main}.

\begin{Corollary}
Let $D$ be a $2$-group as in \eqref{defD}. Then Donovan's Conjecture (cf. \cite{KDonovan}) holds for $2$-blocks of finite groups with defect group $D$.
\end{Corollary}

Combining Theorem~\ref{main} with results in \cite{Sambalemna}, we obtain the following.

\begin{Corollary}
Let $B$ be a $2$-block of a finite group with minimal nonabelian defect groups. Then $B$ satisfies Dade's Projective Conjecture (cf. \cite{Dade2}), Alperin's Weight Conjecture (cf. \cite{Alperinweights}), the Alperin-McKay Conjecture (cf. \cite{Alperinblock}), Brauer's $k(B)$-Conjecture (cf. \cite{BrauerLectures}), Olsson's $k_0(B)$-Conjecture (cf. \cite{OlssonGL}), Eaton's Conjecture (cf. \cite{Eaton}), Brauer's Height-Zero Conjecture (cf. \cite{BrauerLectures}), and the Eaton-Moreto Conjecture (cf. \cite{EatonMoreto}).
\end{Corollary}

We gather together some useful facts about blocks with defect groups as in \eqref{defD}, all of which may be found in or easily deduced from results in \cite{Sambalemna}.

\begin{Lemma}\label{Omnibus}
Let $B$ be a block of a finite group $G$ with defect group $D$ as in \eqref{defD}. Let $(D,b_D)$ be a maximal $B$-subpair. Then
\begin{enumerate}[(i)]
\item $\N_G(D,b_D)$ controls fusion of subpairs contained in $(D,b_D)$,
\item either $B$ is nilpotent or $\lvert\N_G(D,b_D):D\C_G(D)\rvert=3$, and in the latter case $z:=[x,y]$ is the only nontrivial fixed point of $\Z(D)$ under the action of $\N_G(D,b_D)$,
\item if $B$ is not nilpotent, then $\pcore_p(\Z(G)) \leq \langle z \rangle$,
\item if $Q \leq \Z(D)$ and $Q \not\leq D'$, then there is a $B$-subpair $(Q,b_Q)$ with $b_Q$ nilpotent,
\item if $D \in \Syl_2(G)$, then $G$ is solvable.
\end{enumerate}
\end{Lemma}

In our proof of Theorem~\ref{main}, the following result will be very useful.

\begin{Lemma}\label{ExNilBlock}
Let $G$, $B$, $D$ be as in Theorem~\ref{main}. Moreover, let $b$ be a $2$-block of a normal subgroup $H$ of $G$ which is covered by $B$. If a defect group $d$ of $b$ satisfies $|d|<|D|$, then $b$ is nilpotent.
\end{Lemma}

\begin{proof}
It is well-known that $d$ is conjugate to $D\cap H$ (cf. Theorem~E in \cite{Ksolvable}). Replacing $D$ by a conjugate, if necessary, we may assume that $d=D\cap H$. If $d<D$ then also $d\Phi(D)<D$. By Lemma \ref{Omnibus} $B$ has inertial index $t(B)=3$. Since $|D:\Phi(D)|=4$, this implies that $\N_G(D)$ permutes the three maximal subgroups of $D$ transitively. Since $d\Phi(D)$ is normal in $\N_G(D)$, we must have $|D:d\Phi(D)|\ge 4$. But then $d\subseteq\Phi(D)$, and $[\N_H(D),D]\subseteq D\cap H=d\subseteq\Phi(D)$. Thus, $\N_H(D)$ acts trivially on $D/\Phi(D)$. Hence, $\N_H(D)/\C_H(D)$ is a $2$-group. Let $\beta$ be the unique $2$-block of $DH$ covering $b$. Then $D$ is a defect group of $\beta$, by Theorem~E in \cite{Ksolvable}. Let $\beta_D$ be a $2$-block of $D\C_{DH}(D)$ such that $(\beta_D)^{DH}=\beta$. Then $\N_H(D,\beta_D)/\C_H(D)$ and $\N_{DH}(D,\beta_D)/\C_{DH}(D)$ are also $2$-groups, i.\,e. $\beta$ has inertial index $t(\beta)=1$. Since $\beta$ is a controlled block, by Lemma \ref{Omnibus} this implies that $\beta$ is a nilpotent block. But now Proposition~6.5 in \cite{exnilblocks} shows that $b$ is also nilpotent.
\end{proof}

\begin{Corollary}
\label{2PowerIndex}
Let $G$, $B$, $D$ be as in Theorem~\ref{main}. If $H \lhd G$ has index a power of $2$, then $D \leq H$.
\end{Corollary}

\begin{proof}
There is a block $b$ of $H$ covered by $B$ with defect group $D \cap H$. If $D \not\leq H$, then by Lemma \ref{ExNilBlock} $b$ is nilpotent. But then by \cite[6.5]{exnilblocks} $B$ is nilpotent, a contradiction.
\end{proof}

We will apply Lemma~\ref{ExNilBlock} in connection with the results in \cite{exnilblocks}. We are almost in a position to start our proof of Theorem~\ref{main}. First we prove a general result which is presumably well-known, but whose proof we sketch for the convenience of the reader.

\begin{Lemma}
\label{centralproduct}
Let $G=G_1 * G_2$ be a central product of finite groups $G_1$ and $G_2$ and let $B$ be a block of $G$. Let $B_i$ be the (unique) block of $G_i \lhd G$ covered by $B$. Then $B$ is nilpotent if and only if both $B_1$ and $B_2$ are.
\end{Lemma}

\begin{proof}
We may write $G=E/Z$, where $E=G_1 \times G_2$ and $Z \leq \Z(E)$. Let $B_E$ be the unique block of $E$ dominating $B$, so $\pcore_{p'}(Z)$ is in the kernel of $B_E$ and $B_E$ has defect group $D_E$ such that $D_EZ/Z$ is a defect group for $B$. By~\cite[2.6]{ae11odd} $B_E$ is nilpotent if and only if $B$ is. Note that $B_E$ is a product of blocks of $G_1$ and $G_2$ which are nilpotent if and only if $B_1$ and $B_2$ are. Hence it suffices to consider the case $G=G_1 \times G_2$. However, the result follows easily in this case since the normalizer and centralizer of a subgroup $Q$ of $G_1 \times G_2$ are $\N_{G_1}(\pi_1(Q)) \times  \N_{G_2}(\pi_2(Q))$ and $\C_{G_1}(\pi_1(Q)) \times  \C_{G_2}(\pi_2(Q))$, where $\pi_i(Q)$ is the image of the projection onto $G_i$ (we leave the details to the reader).
\end{proof}

\begin{proof}[Proof (of Theorem~\ref{main}).]
We assume that Theorem~\ref{main} fails, and choose a counterexample $G$, $B$, $D$ such that $|G:\Z(G)|$ is as small as possible. Moreover, among all such counterexamples, we choose one where $|G|$ is minimal. Then, by the first Fong reduction, the block $B$ is quasiprimitive, i.\,e. for every normal subgroup $N$ of $G$, there is a unique block of $N$ covered by $B$; in particular, this block of $N$ is $G$-stable. Moreover, by the second Fong reduction $\pcore_{2'}(G)$ is cyclic and central.

We claim that $Q:=\pcore_2(G)\subseteq D'$. Since $Q\unlhd G$ we certainly have $Q\subseteq D$. If $Q=D$ then $B$ has a normal defect group, and $B$ is Morita equivalent to $\mathcal{O}[D\rtimes E]$, by the main result of \cite{Kuelshammer}. Thus, we may assume that $1<Q<D$; in particular, $Q$ is abelian. Let $B_Q$ be a block of $Q\C_G(Q)=\C_G(Q)$ such that $(B_Q)^G=B$. Since $\C_G(Q)\unlhd G$, the block $B_Q$ has defect group $\C_D(Q)$, and either $\C_D(Q)=D$ or $|D:\C_D(Q)|=2$. Since $B$ has inertial index $t(B)=3$, $\N_G(D)$ permutes the maximal subgroups of $D$ transitively. Since $\C_D(Q)\unlhd\N_G(D)$, we must have $\C_D(Q)=D$, i\,e. $Q\subseteq\Z(D)$.

Thus, $B_Q$ is a $2$-block of $\C_G(Q)$ with defect group $D$. If $Q\nsubseteq D'$ then $B_Q$ is nilpotent, by Lemma \ref{Omnibus}. Then, by the main result of \cite{exnilblocks}, $B$ is Morita equivalent to a block of $\N_G(D)$ with defect group $D$, and we are done by the main result of \cite{Kuelshammer}.

This shows that we have indeed $\pcore_2(G)\subseteq D'$; in particular, $\lvert\pcore_2(G)\rvert\le 2$ and thus $\pcore_2(G)\subseteq\Z(G)$. Hence, also $\F(G)=\Z(G)$.

Let $b$ be a block of $\E(G)$ covered by $B$. If $b$ is nilpotent, then, by the main result of \cite{exnilblocks}, $B$ is Morita equivalent to a $2$-block $\widetilde{B}$ of a finite group $\widetilde{G}$ having a nilpotent normal subgroup $\widetilde{N}$ such that $\widetilde{G}/\widetilde{N}\cong G/\E(G)$, and the defect groups of $\widetilde{B}$ are isomorphic to $D$. Thus, by minimality, we must have $\E(G)=1$. Then $\F^*(G)=\F(G)=\Z(G)$, and $G=\C_G(\Z(G))=\C_G(\F^*(G))=\Z(\F^*(G))=\Z(G)$, a contradiction.

Thus, $b$ is not nilpotent. By Lemma~\ref{ExNilBlock}, $b$ has defect group $D$. Let $L_1,\ldots,L_n$ be the components of $G$ and, for $i=1,\ldots,n$, let $b_i$ be a block of $L_i$ covered by $b$. If $b_1,\ldots,b_n$ were nilpotent, then $b$ would also be nilpotent by Lemma \ref{centralproduct}, a contradiction. Thus, we may assume that $b_1$ is a nonnilpotent $2$-block (of the quasisimple group $L_1$). By Lemma~\ref{ExNilBlock}, $D$ is a defect group of $b_1$. But now the following proposition gives a contradiction.
\end{proof}

\begin{Proposition}
Let $D$ be a $2$-group as in $\eqref{defD}$, and let $G$ be a quasisimple group. Then $G$ does not have a $2$-block $B$ with defect group $D$.
\end{Proposition}

Note that the proposition holds for classical groups by~\cite{AnEatonExtraspecial}, where blocks whose defect groups have derived subgroup of prime order are classified. However, since our situation is less general we give new arguments here.

\begin{proof}
We assume the contrary. Then we may also assume that $B$ is faithful. Note that by~\cite{AnEaton} $B$ cannot be nilpotent since $D$ is nonabelian. By Lemma~\ref{Omnibus}, $D$ is not a Sylow $2$-subgroup of $G$, in particular, $64=2^6$ divides $|G|$.

Suppose first that $\overline{G}:=G/\Z(G)\cong A_n$ for some $n\ge 5$. If $\lvert\Z(G)\rvert>2$, then $n\in\{6,7\}$ and $\lvert\Z(G)\rvert\mid 6$, by \cite{classification3}. But then $|G|$ is not divisible by $64$, a contradiction. Thus, we must have $\lvert\Z(G)\rvert\le 2$. Then $\Z(G)\subseteq D$, and $B$ dominates a unique $2$-block $\overline{B}$ of $\overline{G}$ with defect group $\overline{D}:=D/\Z(G)\ne 1$. Let $\mathcal{B}$ be a $2$-block of $S_n$ covering $\overline{B}$. Then $\mathcal{B}$ has a defect group $\mathcal{D}$ such that $\overline{D}\subseteq\mathcal{D}$ and $|\mathcal{D}:\overline{D}|=2$, by results in \cite{JamesKerber}. Let $w$ denote the weight of $\mathcal{B}$. Then, by a result in \cite{JamesKerber}, $\mathcal{D}$ is conjugate to a Sylow $2$-subgroup of $S_{2w}$. We may assume that $\mathcal{D}$ is a Sylow $2$-subgroup of $S_{2w}$. Then $\overline{D}=\mathcal{D}\cap A_n=\mathcal{D}\cap S_{2w}\cap A_n=\mathcal{D}\cap A_{2w}$ is a Sylow $2$-subgroup of $A_{2w}$, and $D$ is a Sylow $2$-subgroup of $A_{2w}$ or $2.A_{2w}$. Thus, $A_{2w}$ is solvable by Lemma~\ref{Omnibus}, so that $w\le 2$ and $|\overline{D}|\le 4$, $|D|\le 8$. Since $|D|\ge 32$, this is a contradiction.

Suppose next that $\overline{G}$ is a sporadic simple group. Then, using Table~1 in \cite{AnEaton}, we get a contradiction immediately unless $G=\text{Ly}$ and $|D|=2^7$. In this remaining case, we get a contradiction since, by \cite{Landrocksporadic}, $D$ is a Sylow $2$-subgroup of $2.A_8$, and $A_8$ is nonsolvable.

Now suppose that $G$ is a group of Lie type in characteristic $2$. Then, by a result of Humphreys \cite{Humphreys}, the $2$-blocks of $G$ have either defect zero or full defect. Thus, again Lemma \ref{Omnibus} leads to a contradiction.

It remains to deal with the groups of Lie type in odd characteristic. We use three strategies to deal with the various subcases.

Suppose first that $\overline{G}\cong\PSL_n(q)$ or $\PSU_n(q)$ where $1<n\in\mathbb{N}$ and $q$ is odd. Except in the cases $\PSL_2(9)$ and $\PSU_4(3)$, there is $E \cong \SL_n(q)$ or $\SU_n(q)$ such that $G$ is a homomorphic image of $E$ with kernel $W$ say. We may rule out the cases $G/\Z(G) \cong \PSL_2(9)$ or $\PSU_4(3)$ using~\cite{gap}. Let $H \cong \GL_n(q)$ or $\GU_n(q)$ with $E \lhd H$. There is a block $B_E$ of $E$ with defect group $D_E$ such that $D_E/W \cong D$. Let $B_H$ be a block of $H$ covering $B_E$ with defect group $D_H$ such that $D_H \cap E = D_E$. Now $B_H$ is labeled by a semisimple element $s \in H$ of odd order such that $D_H \in \Syl_2(\C_H(s))$ (see, for example,~\cite[3.6]{BroueGL}). It follows that $D \in \Syl_2(\C_E(s)/W)$ and so $\C_E(s)/W$ is solvable by Lemma~\ref{Omnibus}. Now $W$ and $H/E$ are solvable, so $\C_H(s)$ is also solvable. By~\cite[1A]{FongGL} $\C_H(s)$ is a direct product of groups of the form $\GL_{n_i}(q^{m_i})$ and $\GU_{n_i}(q^{m_i})$. Write $\C_H(s)\cong \bigl( \prod_{i=1}^{t_1}{\GL_{n_i}(q^{m_i})}\bigr)\times \bigl( \prod_{i=t_1+1}^{t_2}{\GU_{n_i}(q^{m_i})}\bigr)$ where $t_1, t_2 \in\mathbb{N}$, $n_1,\ldots,n_{t_2}\in\mathbb{N}$, and $m_1,\ldots,m_{t_2}\in\mathbb{N}$, with $n_i \geq 3$ for $i > t_1$. Solvability implies that $t_2=t_1$ and that for $i=1,\ldots,t_1$ we have either $n_i=1$ or $n_i=2$, where in the latter case $m_i=1$ and $q=3$. Since $D$, $D_E$, and $D_H$ are nonabelian, we cannot have $n_i=1$ for all $i=1,\ldots,t_1$. Thus, we must have $q=3$ and, w.\,l.\,o.\,g., $n_1=2$, $m_1=1$. Then $D_H$ is a direct product of factors which are either cyclic or isomorphic to $SD_{16}$. Moreover, we have $|D_H:D_E|\le 2$ and $|W|\le 2$. Since $|D:\Phi(D)|=4$, we also have $|D_E:\Phi(D_E)|\le 8$ and $|D_H:\Phi(D_H)|\le 16$.

Suppose first that $|D_H:\Phi(D_H)|=16$. Then $|D_E:\Phi(D_E)|=8$, $|D_H:D_E|=2$, and $|W|=2$. Since $W\nsubseteq\Phi(D_E)$, $D_E\cong D\times W$. If $D_H\cong SD_{16}\times SD_{16}$, then $|D_H|=2^8$ and $|D|=2^6$ which is impossible.

Thus, we must have $D_H\cong SD_{16}\times C_k\times C_l$ where $k$ and $l$ are powers of $2$. Observe that $\Phi(D_E)\subseteq\Phi(D_H)$ and $|D_H:\Phi(D_H)|=16=|D_H:\Phi(D_E)|$. So we must have $\Phi(D_E)=\Phi(D_H)$. Since $\Phi(D_E)\cong\Phi(D)\cong C_{2^{r-1}}\times C_{2^{r-1}}\times C_2$ and $\Phi(D_H)\cong C_4\times C_{k/2}\times C_{l/2}$, this implies that $4=2^{r-1}$, i.\,e. $r=3$ and $\Phi(D)\cong\Phi(D_E)\cong C_4\times C_4\times C_2$. So we may assume that $k=8$, $l=4$. Thus, $D_E\cong D\times C_2$ and $D_H\cong SD_{16}\times C_8\times C_4$. Hence, $D_E'=D'\times 1$, $|D_E'|=2$ and $D_E'\subseteq D_H'\cap\Z(D_H)\cong\Z(SD_{16})\times 1\times 1$, so that $D_E'=\Z(SD_{16})\times 1\times 1$. Moreover, $D_E/D_E'\cong C_8\times C_8\times C_2$ is a subgroup of $D_H/D_E'\cong D_8\times C_8\times C_4$. Hence $\mho_2(C_8\times C_8\times C_2)\cong C_2\times C_2$ is isomorphic to a subgroup of $\mho_2(D_8\times C_8\times C_4)\cong C_2$ which is impossible.

Next we consider the case $|D_H:\Phi(D_H)|=8$. In this case we have $D_H\cong SD_{16}\times C_k$ where $k$ is a power of $2$. Then $\Phi(D_E)\subseteq\Phi(D_H)\cong C_4\times C_{k/2}$ and $\Phi(D)\cong\Phi(D_EW/W)=\Phi(D_E)W/W$. However, this contradicts $\Phi(D)\cong C_{2^{r-1}}\times C_{2^{r-1}}\times C_2$.

The case $|D_H:\Phi(D_H)|\le 4$ is certainly impossible.

Let $G$ be a quasisimple finite group of Lie type with $|G|$ minimised such that there is a block $B$ of $G$ with defect group $D$ as in $\eqref{defD}$. We have just shown that $G$ cannot be of type $A_n(q)$ or $^2A_n(q)$. Details of the following may be found in~\cite{carter} and~\cite{ce}. We may realize $G$ as $\bG^F$, where $\bG$ is a simple, simply-connected group of Lie type defined over the algebraic closure of a finite field, $F:\bG \rightarrow  \bG$ is a Frobenius map and $\bG^F$ is the group of fixed points under $F$. Write $\bG^*$ for the group dual to $\bG$, with corresponding Frobenius map $F^*$. Note that if $\bH$ is an $F$-stable connected reductive subgroup of $\bG$, then $\bH$ has dual $\bH^*$ satisfying $|\bH^F|=|(\bH^*)^{F^*}|$.

Suppose that $G$ is a classical quasisimple group of type $B_n(q)$, $C_n(q)$, $D_n(q)$ or $^2D_n(q)$, where $q>3$ is a power of an odd prime.

By~\cite[1.5]{en08} there is a semisimple element $s \in \bG^*$ of odd order such that $D$ is a Sylow $2$-subgroup of $\bL^F$, where $\bL \leq \bG$ is dual to $\C_{\bG^*}^0(s)$, the connected component of $\C_{\bG^*}(s)$ containing the identity element. By Lemma \ref{Omnibus} $\bL^F$ is solvable. Now by~\cite{ca81} $\C_{\bG^*}(s)$ factorizes as $\bM \bT$, where $\bT$ is a torus and $\bM$ is semisimple, $\C_{(\bG^*)^{F^*}}(s)=\C_{\bG^*}(s)^F = \bM^F\bT^F$ and the components of $\bM^F$ are classical groups defined over fields of order a power of $q$. Hence $\C_{(\bG^*)^{F^*}}(s)$ is either abelian or non-solvable. It follows that $\bL^F$ is either abelian or non-solvable, in either case a contradiction.

{\bf Case 1.} Suppose that $G$ is a quasisimple finite group of Lie type with center of odd order, and further that $q=3$ if $G$ is classical. We analyze $\C_G(z)$, where we recall that $D'=\langle z \rangle$. There is a non-nilpotent block $b_z$ of $\C_G(z)$ with defect group $D$. As $s$ is semisimple, $\C_G(z)$ may be described in detail. By~\cite[4.2.2]{classification3} $\C_G(z)$ has a normal subgroup $C^0$ such that $\C_G(z)/C^0$ is an elementary abelian $2$-group and $C^0 = LT$, where $L=L_1 * \cdots * L_s \lhd C^0$ is a central product of quasisimple groups of Lie type and $T$ is an abelian group acting on each $L_i$ by inner-diagonal automorphisms.

If $G$ is a classical group or any exceptional group of Lie type except $E_6(q)$, $^2E_6(q)$ or $E_7(q)$, then by ~\cite[4.5.1]{classification3} and ~\cite[4.5.2]{classification3} $T$ is a $2$-group. In particular $\C_G(z)/L$ is a $2$-group, so by Corollary \ref{2PowerIndex} $D \leq L$. Let $b_L$ be a block of $L$ covered by $b_z$ with defect group $D$. If $b_L$ is nilpotent, then by \cite[6.5]{exnilblocks} $b_z$ is also nilpotent since $\C_G(z)/L$ is a $2$-group, a contradiction. Hence $b_L$ is not nilpotent. By Lemma \ref{ExNilBlock}, for each $i$ we have that $b_L$ either covers a nilpotent block of $L_i$, or $D \leq L_i$. It follows that either $D \leq L_i$ for some $i$ or $b_L$ covers a nilpotent block of each $L_i$. In the latter case by Lemma \ref{centralproduct} $b_L$ would be nilpotent, a contradiction. Hence $D \leq L_i$ for some $i$ and there is a non-nilpotent block of $L_i$ with defect group $D$. But $|L_i| < |G|$ and $L_i$ is quasisimple, contradicting minimality.

If $G$ is of type $E_6(q)$ or $^2E_6(q)$, then in the notation of~\cite[4.5.1]{classification3} $G$ has (up to isomorphism of centralizers) two conjugacy classes of involutions, with representatives $t_1$ and $t_2$. Suppose first of all that $z$ is of type $t_1$. In this case $\C_G(z)$ has a normal subgroup $X$ of index a power of $2$ such that $X$ is a central product of $L=L_1$ and a cyclic group $A$. Arguing as above, $b_z$ either covers a nilpotent block of $X$, and so is itself nilpotent (a contradiction) or $D \leq X$. So $b_z$ covers a non-nilpotent block $b_X$ of $X$ with defect group $D$. Applying the argument again, either $b_X$ covers nilpotent blocks of $L$ and $A$, in which case $b_X$ would be nilpotent by Lemma \ref{centralproduct} (a contradiction), or $D \leq L$. We have $|L| < |G|$ and $L$ is quasisimple, so by minimality we obtain a contradiction. Consider now the case that $z$ has type $t_2$. Then $\C_G(z)$ has a normal subgroup of index $2$ which is a central product of quasisimple groups, and we can argue as above to again get a contradiction.

If $G$ is of type $E_7(q)$, then in the notation of~\cite[4.5.1]{classification3} $G$ has (up to isomorphism of centralizers) five conjugacy classes of involutions, with representatives $t_1$, $t_4$, $t_4^\prime$, $t_7$ and $t_7^\prime$. In the first three of these cases $T$ is a $2$-group and we may argue exactly as above. In case $t_7$ and $t_7^\prime$, we have $\lvert\C_G(z):C^0\rvert=2$ and by a now familiar argument $D \leq C^0$ and $b_z$ covers a non-nilpotent block of $C^0$ with defect group $D$. There is $X \lhd C^0$ of index $3$ such that $X=LA$, where $L=L_1$ and $A$ is cyclic of order $q \pm 1$. Now by Lemma \ref{Omnibus} $\pcore_2(\Z(A)) = \langle z \rangle$, so $|A|_2 = 2$ and $D \leq L$. By minimality this situation cannot arise since $L$ is quasisimple, and we are done in this case.

{\bf Case 2.} Suppose that $G$ is a quasisimple group of Lie type with center of even order, and further that $q=3$ if $G$ is classical. Note that $G$ cannot be of type $A_n(q)$ or $^2A_n(q)$. Here we must use a different strategy since we may have $\C_G(z)=G$. Let $u \in \Z(D)$ be an involution with $u \neq z$. By Lemma \ref{Omnibus} there is a nilpotent block $b_u$ of $\C_G(u)$ with $b_u^G=B$. As before we refer to~\cite[4.5.2]{classification3} for the structure of $\C_G(u)$, and $\C_G(u) \cong LT$, where $L$ is a central product of either one or two quasisimple groups and $T$ is an abelian group acting on $L$ by inner-diagonal automorphisms. We take a moment to discuss types $D_n(3)$ for $n \geq 4$ even and $^2D_n(3)$. In these two cases the universal version of the group has center of order $4$, and the information given in \cite[4.5.2]{classification3} applies only to the full universal version. In order to extract the required information when $|\Z(G)|=2$ it is necessary to use \cite[4.5.1]{classification3}, taking advantage of the fact that if $Y$ is a finite group, $X \leq \Z(Y)$ with $|X|=2$ and $y \in Y$ is an involution, then $|\C_{Y/X}(yX):\C_Y(y)/X|$ divides $2$. Note also that~\cite[4.5.2]{classification3} gives the fixed point group of an \emph{automorphism} of order $2$ acting on $G$, and that not every such automorphism is realized by an involution in $G$ (this information is contained in the column headed $|\hat{t}|$). We will make no further reference to this fact.

Now $\Z(\C_G(u))$ and $T$ are both $2$-groups, and in each case there is a direct product $E$ of quasisimple groups of Lie type and abelian $2$-groups, with $W \leq \Z(E)$ such that $L \cong E/W$ and $W$ is a $2$-group, and there is a direct product $H$ of finite groups of Lie type such that $E \leq H$ has index a power of $2$ and $H/W$ has a subgroup isomorphic to $\C_G(u)$ of index a power of $2$. Since $W$ and $H/E$ are $2$-groups, by \cite[6.5]{exnilblocks} there are nilpotent blocks $B_E$ of $E$ and $B_H$ of $H$ with defect groups $D_E$ and $D_H$ such that $D_E \leq D_H$ and $D_E/W$ has a subgroup isomorphic to $D$. By Lemma \ref{centralproduct} $B_E$ is a product of nilpotent blocks of finite groups of Lie type, and so by \cite{AnEaton} $D_E$ is abelian. But then $D$ is abelian, a contradiction.
\end{proof}

\begin{Proposition}
Let $B$ be as in Theorem~\ref{main}. Then $D$ is the vertex of the simple $B$-modules.
\end{Proposition}
\begin{proof}
First we consider the situation in the group $D\rtimes E$. Here the three irreducible Brauer characters are linear and can be extended to irreducible ordinary characters. By Theorem~\ref{main} there is a Morita equivalence between $\mathcal{O}[D\rtimes E]$ and $B$. Under this equivalence the three ordinary linear characters map to irreducible characters of height $0$ in $B$. These characters are again extensions of three distinct Brauer characters, since the decomposition matrix is also preserved under Morita equivalence.
Now the claim follows from Theorem~19.26 in \cite{Curtis}.%Green: "On the indecomposable representations of a finite group"
\end{proof}

\section*{Acknowledgment}
The first author is supported by a Royal Society University Research Fellowship, and was partially supported in this research by the “Deutsche Forschungsgemeinschaft”. The third author is also supported by the “Deutsche Forschungsgemeinschaft”.

\end{document}